\documentclass{rae}
\usepackage{amsmath}
\usepackage{amsfonts}
\usepackage{amsthm}
\usepackage{graphicx}

\newtheorem{theorem}{Theorem}
\newtheorem{proposition}[theorem]{Proposition}
\newtheorem{lemma}[theorem]{Lemma}
\newtheorem{corollary}[theorem]{Corollary}
\newtheorem{claim}[theorem]{Claim}
\theoremstyle{remark}
\newtheorem{rem}{Remark}

\DeclareMathOperator{\interior}{int}
\DeclareMathOperator{\cl}{cl}

\newenvironment{rlist}{\begin{list}%
    {\hbox to 0pt{\hss\rm(\roman{enumi})}\unskip\ignorespaces}%
    {\usecounter{enumi}}}%
  {\end{list}}

  {\end{list}}

\def\isparam#1{\let\nextp#1\futurelet\next\isnextbrace}
\def\isnextbrace{\ifcat\bgroup\noexpand\next
  \let\next\nextp
  \else
  \let\next\relax
  \fi
  \next}

\def\real{\mathbb{R}}
\def\R{\mathbb{R}}
\def\set#1#2{\cbr{#1\,:\,#2}}

\def\delimlr#1#2#3{\left#1#2\right#3}
\def\cbr#1{\delimlr\{{#1}\}}

\let\eps\varepsilon
\let\phi\varphi
\def\lset#1{\cbr{#1}}

\title{The Fixed Point of the Composition of Derivatives
\thanks{This research was completed during the Workshop on
Geometric and Dynamical aspects of 
Measure Theory in R\'evf\"ul\"op, which was supported by the Erd\H os Center.}}

\author{M\'arton Elekes, Department of Analysis, E\"otv\"os Lor\'and 
 University, Kecskem\'eti u. 10-12, Budapest, 1053, Hungary;
e-mail: {\tt emarci@cs.elte.hu}
\and Tam\'as Keleti,
\thanks{Partially supported by OTKA grant F 029768} 
Department of Analysis, E\"otv\"os Lor\'and University, 
 Kecskem\'eti u. 10-12, Budapest, 1053, Hungary; e-mail: {\tt elek@cs.elte.hu}
\and 
Vilmos Prokaj,\thanks{Partially supported by OTKA grant F k\'es\H
  obb megmondom a sz\'am\'at}  Department of Probability and Statistics,
 E\"otv\"os Lor\'and University, Kecskem\'eti u. 10-12, Budapest,
 1053,  Hungary;
 e-mail: {\tt prokaj@cs.elte.hu}}


\MathReviews{Primary 26A99; Secondary 26A15, 26B99, 54H25}

\begin{document}
\maketitle

\begin{abstract}
  We give an affirmative answer to a question of K. Ciesielski by showing
  that the composition $f\circ g$ of two derivatives $f,g:[0,1]\to[0,1]$
  always has a fixed point. Using Maximoff's theorem we obtain that the
  composition of two $[0,1]\to[0,1]$ Darboux Baire-1 functions must also 
  have a fixed point.
\end{abstract}

\section*{Introduction} 
In  \cite{GN} R. Gibson and T. Natkaniec mentioned the question of
K. Ciesielski, which asks whether the composition $f\circ g$ of two 
derivatives $f,g:[0,1]\to[0,1]$ must always have a fixed point.
Our main result is an affirmative answer to this question.
(An alternative proof has also been found by M. Cs\"ornyei, 
T. C. O'Neil and D. Preiss, see \cite{CNP}.)

\begin{theorem}\label{thm:fo}
Let  $f$ and $g:[0,1]\to[0,1]$ be derivatives. Then $f\circ g$ has a fixed
point.
\end{theorem}

Since any bounded approximately continuous function is a derivative
(see e.g. \cite{Br}) we get the following:

\begin{corollary}\label{appr}
Let  $f$ and $g:[0,1]\to[0,1]$ be approximately continuous functions. 
Then $f\circ g$ has a fixed point.
\end{corollary}

Now suppose that $f,g:[0,1]\to[0,1]$ are Darboux Baire-1 functions.
Then, by Maximoff's theorem (\cite{Ma}, see also in \cite{Pr}), 
there exist homeomorphisms 
$h,k:[0,1]\to[0,1]$ such that $f\circ h$ and $g\circ k$ are approximately
continuous functions. Then clearly 
$\tilde f = k^{\textrm{-}1}\circ f\circ h$ and 
$\tilde g = h^{\textrm{-}1}\circ g\circ k$ are also $[0,1]\to[0,1]$
approximately continuous functions and $f\circ g$ has a fixed point if and
only if $\tilde f\circ \tilde g$ has a fixed point. Thus from
Corollary~\ref{appr} we get the following:

\begin{corollary}\label{DB1}
Let $f$ and $g:[0,1]\to[0,1]$ be Darboux Baire-1 functions.  Then $f\circ
g$ has a fixed point.
\end{corollary}

\section*{Preliminaries}

We shall denote the square $[\textrm{-}1,2]\times[\textrm{-}1,2]$ by
$Q$. The partial derivatives of a function $h:\R^2\to\R$ with respect to
the $i$-th variable ($i=1,2$) will be denoted by $\partial_i h$. The sets
$\{h=c\}$, $\{h\neq c\}$, $\{h>c\}$, etc.(???) will denote the appropriate
level sets of the function $h$.  The notations $\cl A, \interior A$ and
$\partial A$ stand for the closure, interior and boundary of a set $A$,
respectively. By \emph{component} we shall always mean connected component.

\section{Skeleton of the proof of Theorem~\ref{thm:fo}}\label{skeleton}

Let $f,g:[0,1]\to[0,1]$ be the derivatives of $F$ and $G$, respectively. We
can extend $F$ and $G$ linearly to the complement of the unit interval such
that they remain differentiable and such that the derivatives of the
extended functions are still between 0 and 1. Therefore we may assume that
$f$ and $g$ are derivatives defined on the whole real line and $0\le f,g\le
1$.

The key step of the proof is the following. Let
\begin{equation}\label{formula}
H(x,y)=F(x)+G(y)-xy\ \ \ (\ (x,y) \in\real^2\ ),
\end{equation}
\begin{center}
$\textrm{where } F'=f,\ G'=g:\real\to [0,1] \textrm{ are derivatives}.$
\end{center}
Then $H$ is differentiable as a function of two variables, and its gradient
is
\begin{equation}\label{gradient}
H'(x,y) = (f(x)-y,g(y)-x).
\end{equation} 
Thus the gradient vanishes at $(x_0,y_0)$ if and only if $x_0$ is a fixed
point of $g\circ f$ and $y_0$ is a fixed point of $f\circ g$. Moreover, the
gradient cannot be zero outside the closed unit square, since $0\le f,g\le
1$. Therefore it is enough to prove that there exists a point in $\real^2$
where the gradient vanishes.

We argue by contradiction, so throughout the paper we shall suppose that
the gradient of $H$ nowhere vanishes.

Let us now examine the behavior of $H$ on the edges of $Q$. By
(\ref{gradient}) the partial derivative $\partial_1 H$ is clearly negative
on the top edge of $Q$ and positive on the bottom edge, and similarly
$\partial_2 H$ is positive on the left edge and negative on the right
edge. This implies the following.

\begin{lemma}\label{monotone}
$H$ is strictly decreasing on the top and right edges, and strictly
increasing on the bottom and left edges of $Q$.
\end{lemma}

Consequently,
\[
\max\left(H(\textrm{-}1,\textrm{-}1),H(2,2)\right)<
\min\left(H(\textrm{-}1,2), H(2,\textrm{-}1)\right).
\]
This inequality suggests that there must be a kind of `saddle point' in
$Q$. Therefore define
\[
c=\inf\set{d}{(\textrm{-}1,\textrm{-}1) \text{ and } (2,2) \text{ are in
the same component of } \{H\leq d\} \cap Q }
\]
as a candidate for the value of $H$ at a saddle point.
At a typical saddle point $q$ of some smooth function $K:\real^2\to\real$
such that $K(q)=d$, the level set in a neighborhood of $q$ consists of two
smooth curves intersecting each other at $q$. In Section~\ref{proof:c} we
shall prove the following theorem, which states that the level set
$\{H=c\}$ behaves indeed in a similar way.

\begin{theorem}\label{c}
$\{H=c\} \cap Q$ intersects each edge of $Q$ at exactly one point (which is
not an endpoint of the edge). Moreover, this level set cuts $Q$ into four
pieces, that is $\{H\neq c\} \cap Q$ has four components, and each of these
components contains one of the corners of $Q$. The values of $H$ are
greater than $c$ in the components containing the top left and the bottom
right corner of and less than $c$ in the components containing the bottom left
and the top right corner.
\end{theorem}

This theorem suggests the picture that the level set $\{H=c\} \cap Q$ looks
like the union of two arcs crossing each other, one from the left to the
right and one from the top to the bottom.
We are particularly interested in the existence of this `crossing point',
because at such a point four components of $\{H\neq c\} \cap Q$ should
meet. Indeed, we shall prove the following (see
Section~\ref{proof:branching}).

\begin{theorem}\label{branching}
There exists a point $p \in \{H=c\} \cap \interior Q$ that is in the
closure of more than two components of $\{H\neq c\} \cap Q$.
\end{theorem}

\begin{rem}
It is in fact not much harder to see that there exists a point $p \in
\{H=c\} \cap \interior Q$ in the intersection of the closure of exactly
four components.
\end{rem}

But on the other hand we shall also prove (see
Section~\ref{proof:non-branching}) the following theorem about the local
behavior of the level set.

\begin{theorem}\label{non-branching}
Every $p \in \{H=c\} \cap \interior Q$ has a neighborhood that intersects
exactly two components of $\{H\neq c\} \cap Q$.
\end{theorem}

However, the last two theorems clearly contradict each other, which will
complete the proof of Theorem~\ref{thm:fo}.

\section{Proof of Theorem~\ref{c}}\label{proof:c}

Recall that we defined
\[
c=\inf\set{d}{(\textrm{-}1,\textrm{-}1) \text{ and } (2,2) \text{ are in
the same component of } \{H\leq d\} \cap Q }.
\]

\begin{lemma}\label{min}
In fact, $c$ is a minimum; that is, 
$(\textrm{-}1,\textrm{-}1) \text{ and } (2,2)$ are in
the same component of $\{H\leq c\} \cap Q$.
\end{lemma}
\begin{proof}
Let $K_n$ be the component of
$\lset{h\leq c+1/n} \cap Q$ containing both $(\textrm{-}1,\textrm{-}1)$ and $(2,2).$ 
Then $K_n$ is clearly a decreasing sequence of compact connected sets.
A well known theorem (see e.g. \cite[I. 9. 4]{Wh}) states 
that the intersection of such a sequence is connected. 
Hence $\cap_n K_n\subset 
\lset{h\leq c}$ is connected and contains both $(\textrm{-}1,\textrm{-}1)$ and
$(2,2)$, which shows that
$(\textrm{-}1,\textrm{-}1) \text{ and } (2,2)$ are in
the same component of $\{H\leq c\} \cap Q$.
\end{proof}

\begin{lemma}
$\max\lset{H(\textrm{-}1,\textrm{-}1),H(2,2)}<c<\min\lset{H(\textrm{-}1,2),H(2,\textrm{-}1)}.$
\end{lemma}
\begin{proof}
To show $H(2,2)<c$ consider the polygon $P$ with vertices
$(\textrm{-}1,2)$, $(1,2)$, $(2,1)$ and $(2,\textrm{-}1)$. 
Since $\partial_1 H(x,y)=f(x)-y<0$ if $y>1$ and 
$\partial_2 H(x,y)=g(y)-x<0$ if $x>0$, it is easy to check that
$H$ is strictly bigger than $H(2,2)$ on the whole $P$. 
Therefore $P$ connects $(\textrm{-}1,2)$ and $(2,\textrm{-}1)$ in $\lset{H>H(2,2}$,
thus $(\textrm{-}1,\textrm{-}1)$ and $(2,2)$ cannot be in the same component of 
$\lset{H\le H(2,2)}$. By the previous Lemma this implies that
$H(2,2)<c$. Proving that $H(\textrm{-}1,\textrm{-}1)$ is similar.

To show that $c<H(\textrm{-}1,2)$ consider the polygon $P'$ 
with vertices $(\textrm{-}1,\textrm{-}1)$, $(\textrm{-}1,1)$, $(0,2)$ and 
$(2,2)$. Like above one 
can show that $H$ is strictly less than $H(\textrm{-}1,2)$ on the whole $P'$,
so denoting the maximum
of $H$ on $P'$ by $d$, we have $d<H(\textrm{-}1,2)$. 
Since $P'$ connects $(\textrm{-}1,\textrm{-}1)$
and $(2,2)$ in $\{H\le d\}\cap Q$ we also have $c\le d$, therefore
we have $c<H(\textrm{-}1,2)$. Proving that $c<H(2,\textrm{-}1)$ would be similar.
\end{proof}

The previous Lemma together with Lemma~\ref{monotone} clearly implies
the first statement of Theorem~\ref{c}: $\lset{H=c}\cap Q$ intersect 
each edge of $Q$ at exactly one point (which is not an endpoint of
the edge). Since $4$ points cut $\partial Q$ into four components
we also get that $\lset{H\neq c}\cap \partial Q$ has $4$ components
and each of them contains a vertex of $Q$. Therefore for completing
the proof of Theorem~\ref{c} we have to show that any component
of $\lset{H\neq c}\cap Q$ intersect $\partial Q$ and that 
the vertices of $Q$ belong to different components of 
$\lset{H\neq c}\cap Q$. 

The first claim is clear since if $C$ were a component of 
$\lset{H\neq c}\cap Q$ inside $Q$ then one of the (global) extrema 
of $H$ on the (compact)
closure of $C$ could not be on the boundary of $C$ (where $H=c$),
so this would be a local extremum but a function with non-vanishing
gradient cannot have a local extremum.

To prove the second claim first note that the components of 
$\lset{H\neq c}\cap Q$ are the components of $\lset{H<c}\cap Q$
and the components of $\lset{H>c}\cap Q$.
By the previous Lemma,
$(\textrm{-}1,\textrm{-}1)$ and $(2,2)$ are in $\lset{H<c}\cap Q$ while
$(\textrm{-}1,2)$ and $(2,\textrm{-}1)$ are in $\lset{H>c}\cap Q$, so it is enough to
prove that that the opposite vertices belong to different components.

Assume that $(\textrm{-}1,\textrm{-}1)$ and $(2,2)$ are in the
same component of $\lset{H<c}\cap Q$. Then there is a continuous curve 
in $\lset{H<c}\cap Q$ that connects $(\textrm{-}1,\textrm{-}1)$ and $(2,2).$ 
Let $d$ be the maximum of $H$ on this curve. Then on one hand
$d<c$; on the other hand,
$(\textrm{-}1,\textrm{-}1) \text{ and } (2,2)$ are in
the same component of $\{H\leq d\} \cap Q$,
so $c\leq d$, which is a contradiction.

Finally, assume that $(\textrm{-}1,2)$ and $(2,\textrm{-}1)$ are in the same components of 
$\lset{H>c}\cap Q$. Then there is a continuous curve 
in $\lset{H>c}\cap Q$ that connects $(\textrm{-}1,2)$ and $(2,\textrm{-}1)$, so it
also separates $(\textrm{-}1,1)$ and $(2,2)$. But this is a contradiction
with Lemma~\ref{min}.

\section{Proof of Theorem~\ref{branching}}\label{proof:branching}

The following topological lemma is surely well known. However, we were
unable to find it in the literature, so we sketch a proof here.

\begin{lemma}\label{Zoretti}
Let $M$ be a closed subset of $Q$ and $p\neq q\in M\cap\partial Q.$ Denote
by $E_1$ and $E_2$ the two connected components of $\partial
Q\setminus\{p,q\}$. Then the following statements are equivalent:
\begin{itemize}
\item[(i)] $p$ and $q$ are not in the same component of $M$.
\item[(ii)] there are points $e_i\in E_i \ (i=1,2)$ that can be connected
by a polygon in $Q\setminus M$.
\end{itemize}
\end{lemma}

\begin{proof}
The implication (ii) $\Rightarrow$ (i) is obvious. The implication (i)
$\Rightarrow$ (ii) easily follows from Zoretti's Theorem (see \cite{Wh}
Ch. VI., Cor. 3.11), which states that if $K$ is a component of a compact
set $M$ in the plane and $\eps>0$, then there exists a simple closed
polygon $P$ in the $\eps$-neighborhood of $K$ that encloses $K$ and is
disjoint from $M$.
\end{proof}

Now we turn to the proof of Theorem~\ref{branching}. Let us denote by $C_p$
the component of $\{H\neq c\} \cap Q$ that contains the corner $p$ of $Q$
(Theorem~\ref{c} shows that $p\in\{H \neq c\} \cap Q$). Again by
Theorem~\ref{c} we have $\{H\neq c\} \cap Q = C_{(\textrm{-}1,\textrm{-}1)}
\cup C_{(2,2)} \cup C_{(\textrm{-}1,2)} \cup C_{(2,\textrm{-}1)}$, and $H$
is less than $c$ in the first two and greater than $c$ in the last two
components.

\begin{proposition}\label{r}
There exists a point $r \in \cl{C_{(\textrm{-}1,\textrm{-}1)}} \cap
\cl{C_{(2,2)}}$.
\end{proposition}

\begin{proof}
Suppose, on the contrary, that there is no such point. Then the components
of $\{H\leq c\}\cap Q$ are $\cl{C_{(\textrm{-}1,\textrm{-}1)}}$ and
$\cl{C_{(2,2)}}$, hence $(\textrm{-}1,\textrm{-}1)$ and $(2,2)$ are in
different components. Therefore, if we apply the
previous lemma to $\{H\leq c\}\cap Q$ with
$p=(\textrm{-}1,\textrm{-}1)$ and $q=(2,2)$, then we obtain a polygon in
$\{H>c\} \cap Q$ joining either the left or the top edge to either the
right or the bottom edge of $Q$.  Let us now join the endpoints of this
polygon to the points $(\textrm{-}1,\textrm{-}1)$ and $(2,2)$ by two
segments. By Lemma~\ref{monotone} these segments must also be in $\{H > c\}
\cap Q$, and we thus obtain a polygon in $\{H > c\} \cap Q$ connecting
$(\textrm{-}1,\textrm{-}1)$ and $(2,2)$, which contradicts Theorem~\ref{c}.
\end{proof}

Our next aim is to show that such a point $r$ cannot be on the edges of
$Q$. The other cases being similar we only consider the case of the left
edge.

\begin{lemma}
$\{[\textrm{-}1,0) \times [\textrm{-}1,2]\} \cap C_{(2,2)} = \emptyset$
\end{lemma}

\begin{proof}
We have to prove that if $(x,y)\in [\textrm{-}1,0) \times [\textrm{-}1,2])$
such that $H(x,y)<c$, then $(x,y)\in C_{(\textrm{-}1,\textrm{-}1)}$. In
order to do this, it is sufficient to show that the vertical segment $S_1$
between $(x,y)$ and $(x,\textrm{-}1)$ and the horizontal segment $S_2$
between $(x,\textrm{-}1)$ and $(\textrm{-}1,\textrm{-}1)$ together form a
polygon in $\lset{H(x,y)<c}$ connecting $(x,y)$ to
$(\textrm{-}1,\textrm{-}1)$. But by (\ref{gradient}) in
Section~\ref{skeleton} the partial derivative $\partial_2 H$ is positive on
$S_1$, thus all values of $H$ are less then $H(x,y)$ here, hence less then
$c$, and $\partial_1 H$ is positive on $S_2$, thus the values of $H$ here
are less then $H(x,\textrm{-}1)$, which is again less than $c$.
\end{proof}

Let $r$ be a point provided by Proposition~\ref{r}. If we repeat the
agrument of the previous lemma for the other three edges, then we obtain
that $r$ must be in the interior of $Q$. To complete the proof of
Theorem~\ref{branching} it is enough to show that there exists a component
in which $H$ is greater than $c$ (that is either component
$C_{(\textrm{-}1,2)}$ or $C_{(2,\textrm{-}1)}$) such that the closure of
the component contains $r$. But $r\in \interior Q$, and then this statement
follows from the fact that the gradient of $H$ nowhere vanishes, so $H$
attains no local extremum.

\section{Proof of Theorem~\ref{non-branching}}\label{proof:non-branching}

\begin{lemma}\label{teglalap} 
Let $x_1,x_2,y_1,y_2\in\real$ be such that $x_1<x_2$
and $y_1<y_2$. Then $H(x_1,y_2)+H(x_2,y_1)-H(x_1,y_1)-H(x_2,y_2)>0$. 
\end{lemma}

\begin{proof}
Using the definition (\ref{formula}) of $H$, a straightforward computation 
shows that the value of the above sum in fact equals 
$(x_2-x_1)(y_2-y_1)>0.$
\end{proof}

\begin{lemma}\label{Tamas}
Let $(x_0,y_0)\in\real^2$ be such that $H(x_0,y_0)=0$ and $\partial_2
H(x_0,y_0)\neq 0$. Then there exists $\eps>0$ such that if
$x_1\in[x_0-\eps,x_0+\eps]$ and $y_1,
y_2\in[y_0-\eps,y_0+\eps]$ and $H(x_1,y_1)=H(x_1,y_2)=0$ then one
of the followings holds:
\begin{rlist}
\item $y_1=y_2=y_0$,
\item $y_1<y_0$ and $y_2<y_0$,
\item $y_1>y_0$ and $y_2>y_0$.
\end{rlist}
\end{lemma}
\begin{proof}
From $H(x_0,y_0)=0$ and (\ref{formula})
we get $G(y_0)=x_0 y_0 - F(x_0)$, while  
$\partial_2 H(x_0,y_0)\neq 0$ means 
$G'(y_0)\neq x_0$. These imply that for a
sufficiently small $\eps$ if $y_0-\eps<y_1\le y_0\le y_2<y_0+\eps$ and 
$y_1\neq y_2$ then the slope of the segment between $(y_1,G(y_1))$ and
$(y_2,G(y_2))$ is not in $[x_0-\eps,x_0+\eps]$.
But if $H(x_1,y_1)=H(x_1,y_2)=0$ then $G(y_1)=x_1 y_1 - F(x_1)$ and 
$G(y_2)=x_1 y_2 - F(x_1)$, so the slope of the segment between 
$(y_1,G(y_1))$ and $(y_2,G(y_2))$ is $x_1$. 
Therefore, with this $\eps$, the conditions of the lemma imply that 
one of (i),(ii) or (iii) must hold.
\end{proof}

Before turning to the proof of Theorem~\ref{non-branching} 
observe that if 
$H$ is given by equation (\ref{formula}) then $H_1(x,y)=-H(1-x,y)$ is
also of the form (\ref{formula}).
Indeed, taking $F_1(x)=-F(1-x)$ and $G_1(y)=y-G(y)$, we get
$f_1(x)=F'_1(x)=f(1-x)$ and $g_1(y)=G'_1(y)=1-g(y) : [0,1]\to[0,1]$
derivatives and $H_1(x,y)=F_1(x)+G_1(x)-xy= -H(1-x,y).$ 
It is also clear if $H$ is of the form (\ref{formula}) then
$H_2(x,y)=H(y,x)$ is also of the form (\ref{formula}). This means
that when we study the local behavior of the function $H$ around an
arbitrary point $(x,y)\in\lset{H=0}\cap Q$ we can assume that 
$\partial_2 H(x,y)>0$ as this is
true (at the apropriate point) 
for at least one of the functions  we can get using these
symmetries.
By adding a constant to $H$ we can also assume that $d=0$.
Therefore for proving Theorem~\ref{non-branching}
we can assume these restrictions,
so it is enough to prove the following.

\begin{claim}\label{two}
Let $(x_0,y_0)\interior Q$ be such that $H(x_0,y_0)=0$ and 
$\partial_2 H(x_0,y_0)>0$. 
Then there exists a
neighbourhood of $(x_0,y_0)$ that intersects exactly two components of
$\{H\neq 0\}\cap Q$.
\end{claim}
\begin{proof}


Every neighbourhood of $(x_0,y_0)$ 
has to meet at least two components, otherwise
$(x_0,y_0)$ would be a local extremum but we assumed 
that gradient of $H$ is nowhere vanishing.

In order to show that a small neighborhood of $(x_0,y_0)$ meets at most
two components it is enough to construct a polygon in $Q$
around the point which intersects $\{H=0\}$ at exactly two points. Indeed, 
in this case there can be at most two components that intersect the 
polygon  
(since two points cut a polygon into two components).
On the other hand, every
component of $\{H\neq 0\}\cap Q$ that intersects the interior of the polygon 
has to 
meet the
polygon itself, 
since by Theorem~\ref{branching} each component of $\{H\neq 0\}\cap Q$
contains a vertex of $Q$. 

To construct such
a polygon it is enough to find two rectangles in $Q$, one to the left and 
one to
the right from $(x_0,y_0)$, which contain this point on vertical edges (but
not as a vertex) and which both intersect $\{H=0\}$ in just one additional
point. (See Fig. 2.)

\begin{figure}[!ht]
\begin{center}
\includegraphics[scale=.8]{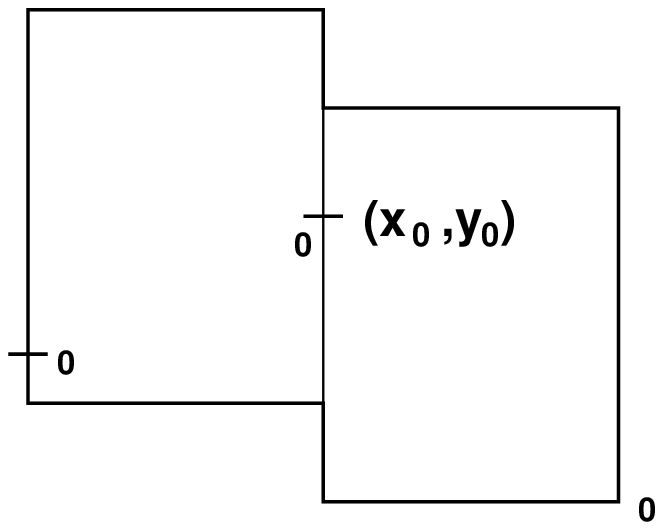}
\caption{}\label{fig:polygon}
\end{center}
\end{figure}

Since  $\partial_2 H(x_0,y_0)>0$ we can find a small 
open sector in $Q$
with vertex $(x_0,y_0)$ with vertical axis upwards, where $H$ is 
positive and a similar negative sector
downwards. (See Fig. \ref{fig:harom}. and Fig. \ref{fig:negy}.)
Choose the radius of the sectors smaller than the $\eps$ we get in
Lemma~\ref{Tamas}. Choose $x_1<x_0$ and $x_2>x_0$ such that
the vertical lines through them intersect the two sectors.

\begin{figure}[!ht]
\begin{center}
\includegraphics[scale=.8]{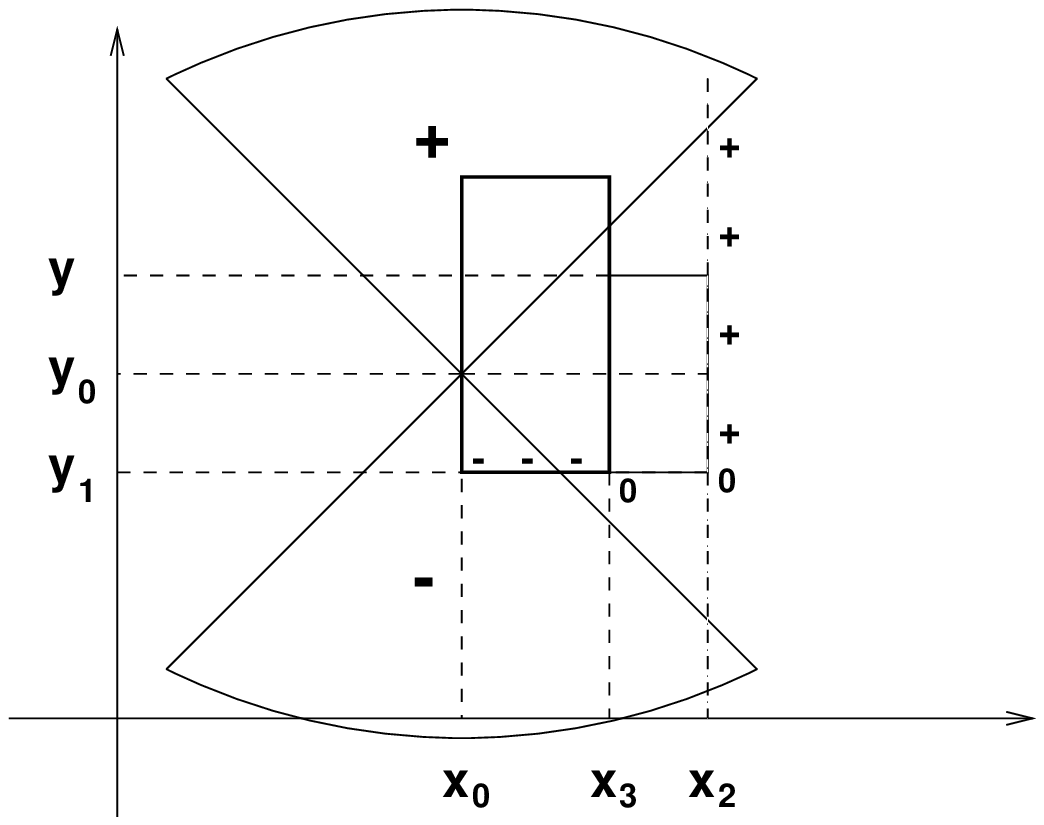}
\caption{}\label{fig:harom}
\end{center}
\end{figure}

First we construct the rectangle on the right hand side. (See Fig. \ref{fig:harom}.)
If (i) holds in
Lemma \ref{Tamas} for the point $(x_0,y_0)$ and for $x_2$ then we are done,
so as the other two cases are similar we assume that (ii) holds. This means
that even for the maximum $y_1$ of the zero set of $H$ between the two
sectors on this vertical line, $y_1<y_0$ holds. 
We choose the leftmost point $(x_3,y_1)$ of the zero set on the segment 
$[x_0,x_2]\times\{y_1\}$ as the lower right corner of the rectangle and we
choose the upper right corner from the upper sector.
Clearly there are only negative values on the bottom edge, so
what remains to show is that all values are positive above $(x_3,y_1)$ on
the vertical line until it reaches the upper sector. If $x_3=x_2$ then we
are done, otherwise this is an easy consequence of Lemma \ref{teglalap}
once we apply it to $x_3,x_2,y_1$ and any number $y_2>y_1$ such that
$(x_3,y_2)$ is still between the two sectors. 

\begin{figure}[!ht]
\begin{center}
\includegraphics[scale=.8]{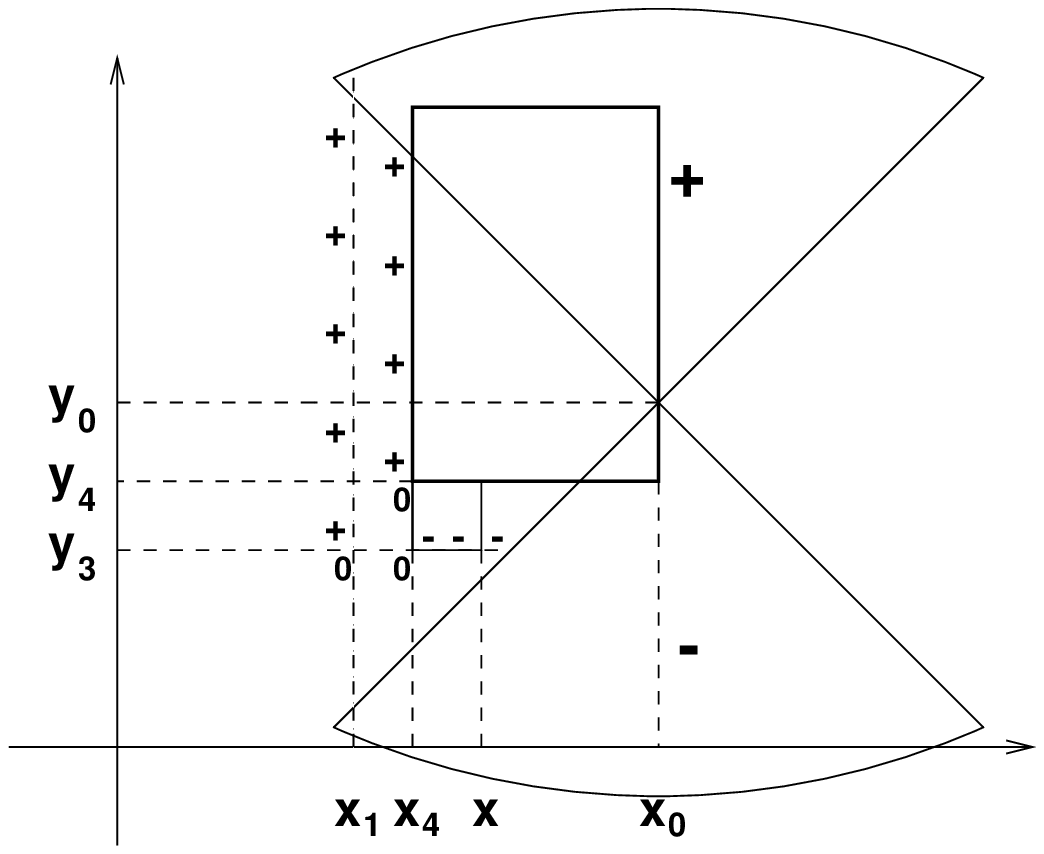}
\caption{}\label{fig:negy}
\end{center}
\end{figure}

Now we turn to the left hand side. (See Fig. \ref{fig:negy}.)
We assume again that (ii) of
Lemma~\ref{Tamas} holds for
$x_1$, and choose a maximal $y_3$ such that $(x_1,y_3)$ is between the two
sectors and $h(x_1,y_3)=0$. 
Note that, by (ii), we have $y_3<y_0$. 
Let $(x_4,y_3)$ be the rightmost point of the zero set of $H$ on the segment
$[x_1,x_0]\times\{y_3\}$. 
Let us now define the lower left corner $(x_4,y_4)$
of the rectangle as the point of maximal $y$ coordinate on the vertical
line $x=x_4$ between the sectors, where $H$ vanishes. 
By Lemma \ref{Tamas}, $y_4<y_0$
and clearly all values on the vertical half-line are positive until it
reaches the upper sector. Thus we only have to show that $H(x,y_4)<0$ for
all $x_4<x<x_0$, which easily follows from Lemma \ref{teglalap} once we
apply it to $x_4,x,y_3$ and $y_4$. 
\end{proof}

\begin{rem}
Theorem~\ref{non-branching} is not true for every $\R^2\to\R$ 
differentiable function with nowhere vanishing gradient. For an example
see \cite{Bu}.
\end{rem}


\begin{thebibliography}{9}

\bibitem{Br} A. Bruckner, 
\emph{Differentiation of Real Functions}, CRM Monograph Series, 1994.

\bibitem{Bu} Z. Buczolich, 
\emph{Level sets of functions $f(x,y)$ with non-vanishing gradient},
Journ. Math. Anal. Appl. {\bf 185} (1994), 27-35.

\bibitem{CNP} M. Cs\"ornyei, T. C. O'Neil, D. Preiss,
\emph{The composition of two derivatives has a fixed point},
to appear in Real Anal. Exchange.

\bibitem{GN} R. Gibson, T. Natkaniec, 
\emph{Darboux like functions. Old problems and new results},
Real Anal. Exchange {\bf 24} (2) (1998-99), 487-496.

\bibitem{Ma} I. Maximoff,
\emph{Sur la transformation continue des fonctions}
Bull. Soc. Phys. Math. Kazan {\bf 12} (1940), no. 3, 9-41 
(Russian, French summary).

\bibitem{Pr} D. Preiss,
\emph{Maximoff's Theorem},
Real Anal. Exchange {\bf 5} (1) (1979-80), 92-104.

\bibitem{Wh} G. T. Whyburn, 
\emph{Analytic Topology},
AMS Colloquium Publications Vol. 28, 1942.
\end{thebibliography}
\end{document}